\documentclass[a4paper,10pt]{article}

\usepackage{amsmath, amstext, amsopn, amsthm, amscd, amsxtra, amsfonts, amssymb, amsbsy}
\usepackage{paralist}
\usepackage[mathcal]{euscript}

\renewcommand{\epsilon}{\varepsilon}
\renewcommand{\bar}{\overline}

\newcommand{\ind}[1]{\operatorname{ind}_{#1}\nolimits}
\newcommand{\proj}[1]{\operatorname{proj}_{#1}\nolimits}
\newcommand{\Projeins}{\ensuremath{\operatorname{Proj}^{\,\operatorname{1}}}}
\newcommand{\Bigcap}[1]{\ensuremath{\mathop{\cap}_{#1}}}
\newcommand{\Bigcup}[1]{\ensuremath{\mathop{\cup}_{#1}}}

\newtheorem*{thmA}{Theorem A}
\newtheorem*{thmB}{Theorem B}
\newtheorem{thm}{Theorem}[section]
\newtheorem{prop}{Proposition}[section]
\newtheorem{cor}{Corollary}[section]
\newtheorem{lem}{Lemma}[section]

\begin{document}
\title{Bornological projective limits\\of inductive limits of normed spaces}
\author{Jos\'{e} Bonet  and Sven-Ake Wegner}
\date{May 1, 2011}
\maketitle
\begin{center}\vspace{-7pt}
\textit{Dedicated to the memory of Susanne Dierolf.}
\end{center}
\vspace{3pt}
\begin{center}
\begin{minipage}[c]{11cm}\small\textbf{Abstract:}~We establish a criterion to decide when a countable projective limit of countable inductive limits of normed spaces is bornological. We compare the conditions occurring within our criterion with well-known abstract conditions from the context of homological algebra and with conditions arising within the investigation of weighted PLB-spaces of continuous functions.
\end{minipage}
\end{center}

\vspace{5pt}

\section{Introduction}

Many areas of the modern theory of locally convex spaces
which has been successful in the recent solution of analytic problems gained
great insight with new techniques related to homological algebra. In
particular, the derived projective limit functor, introduced first
by Palamodov \cite{Palamodov1968,Palamodov1971}, and studied since the mid 1980's by Vogt
\cite{VogtLectures} and others, played a very important role and became a very
useful tool. An excellent presentation of the homological tools can
be found in the book by Wengenroth \cite{Wengenroth}.  
Vogt \cite{VogtLectures, Vogt1989} was the first one to notice that the
vanishing of the derived projective limit functor for a countable
spectrum of LB-spaces is related to the locally convex properties
of the projective limit of the spectrum (for example being barrelled
or bornological); see Theorems 3.3.4 and 3.3.6 in \cite{Wengenroth}. He also
gave complete characterizations in the case of sequence spaces in
\cite[Section 4]{Vogt1989}.
\smallskip
\\For projective spectra of LB-spaces the vanishing of the functor $\Projeins$ is a sufficient condition for the corresponding projective limit to be ultrabornological (and thus also barrelled). 
A countable projective limit of countable inductive limits of Banach spaces is called a \textit{PLB-space}. 
PLB-spaces constitute a class which is strictly larger than the
class of PLS-spaces. A locally convex space is a PLS-space if it
is a countable projective limit of DFS-spaces (i.e.\ of countable
inductive limits of Banach spaces with compact linking maps). The
class of PLS-spaces contains many
natural examples from analysis like the space of distributions, the
space of real analytic functions and several spaces of
ultradifferentiable functions and ultradistributions. In recent years,
this class has played a relevant role in the
applications of abstract functional analysis to linear problems in
analysis. These problems include the solvability, existence of
solution operators and parameter dependence of linear partial
differential operators and convolution operators, the linear
extension of infinitely differentiable, holomorphic or real analytic
functions, and the study of composition operators on spaces of real
analytic functions, among other topics. See the survey article of
Doma\'{n}ski \cite{Domanski2004}. As can be
observed in chapter 5 of Wengenroth's lecture notes \cite{Wengenroth}, the study of the
splitting of short exact sequences of Fr\'echet or more general
spaces requires the consideration of PLB-spaces which are not PLS-spaces.
There are several possibilities to conclude that $\Projeins=0$ holds for projective spectra of LB-spaces. For a concrete projective limit, it firstly depends on abstract properties of the spectrum (like being reduced or having compact linking maps) whether a stronger or a weaker condition can be used.
\smallskip
\\The main result of this article Theorem \ref{THM} is a criterion to decide when a countable projective limit of countable inductive limits of normed spaces is bornological, that constitutes an extension of the methods for LB-spaces mentioned above. It can be used as a criterion for (quasi-)barrelledness of projective limits of LB-spaces which have a dense topological subspace which is the projective limit of inductive limits of normed spaces. In fact, our main motivation to prove Theorem \ref{THM} was to treat weighted spaces of polynomials and weighted spaces of continuous functions with compact support. The study of projective limits of weighted inductive limits of spaces of polynomials was necessary to investigate when a weighted PLB-space of holomorphic functions is barrelled in cases when the projective limit functor cannot be directly applied.
Results on this subject will be contained in a forthcoming paper by S.-A.\ Wegner. See also the last named author's doctoral thesis \cite{DISS}. In Section \ref{Applications} of this paper we present applications to weighted PLB-spaces of continuous functions. These spaces were investigated in \cite{ABB2009} and they contain not only the sequence spaces defined with sup-norms, but also permit one to treat spaces of
continuous linear operators from a K\"othe echelon space into
another or tensor products of Fr\'echet and LB-spaces of null
sequences. In the case of weighted PLB-spaces of continuous functions,  the dense subspace and its representation as a projective limit of inductive limits of normed spaces arise very naturally. Using this representation  we give an alternative (non-homological) proof of a result of Agethen, Bierstedt, Bonet \cite{ABB2009} in the case of functions vanishing at infinity. The situation in the case of bounded functions is the following: Agethen, Bierstedt, Bonet \cite{ABB2009} proved with the help of $\Projeins$ that a certain  condition on the weights is sufficient for ultrabornologicity. But it follows from their results that this condition cannot be necessary and that a necessary condition cannot be found using $\Projeins$, cf.~Theorem B in Section \ref{Applications}. We explain why our criterion does not yield a solution to this problem, either. The latter follows from a comparison of the condition appearing in Theorem \ref{THM} with \textquotedblleft{}classical\textquotedblright{} $\Projeins$-conditions 
which we perform in Section \ref{Criterion}. At the end of Section \ref{Applications} we extend this comparison including weight conditions used by Agethen, Bierstedt, Bonet \cite{ABB2009}.
\smallskip
\\We refer the reader to \cite{BMS1982} for weighted spaces of continuous functions and to \cite{Jarchow,  KoetheI, KoetheII,MeiseVogtEnglisch, BPC} for the general theory of locally convex spaces.

\section{A criterion for the bornologicity of projective limits of inductive limits of normed spaces}\label{Criterion}

In the sequel we let $\mathcal{X}=(X_N,\rho_M^N)$ denote a projective spectrum of inductive limits of normed spaces $X_N=\ind{n}X_{N,n}$, where we use the notation of Wengenroth \cite[Definition 3.1.1]{Wengenroth} and \textit{assume in addition that the $\rho_M^N$ are inclusions of linear subspaces}. Denote by $X=\proj{N}\ind{n}X_{N,n}$ the limit of the spectrum $\mathcal{X}$ and by $B_{N,n}$ the closed unit ball of the normed space $X_{N,n}$. For all $N$ we assume that for each bounded set $B\subseteq X_N$ there exists $n$ such that $B\subseteq B_{N,n}$. This assumption is equivalent to the fact that the spaces $X_N$ are regular inductive limits of normed spaces. We keep this notation in the rest of the section.

\begin{lem}\label{LEM-1} Let $X=\proj{N}\ind{n}X_{N,n}$ be a projective limit of regular inductive limits  of normed spaces. Assume that
$$
\text{(B1)}\;\;\;\;\;\;\;\forall\:N\;\exists\:M\;\forall\:m\;\exists\:n\;\colon B_{M,m}\subseteq \Bigcap{k\in\mathbb{N}}(B_{N,n}\cap X + {\textstyle\frac{1}{k}}B_{N,n})
$$
holds for the spectrum $\mathcal{X}$. Let $T\subseteq X$ be an absolutely convex set. Then
$$
\text{(B2)}\;\;\;\;\;\;\;\exists\:N\;\forall\:n\;\exists\:S>0\;\colon B_{N,n}\cap X\subseteq ST\;\;\;\;\;\;\;\;\;\;\;\;\;\;\;\;\;\;\;\;\;\;\;\;\;\;\;\;\;\;
$$
holds if and only if $T$ is a 0-neighborhood in $X$.
\end{lem}
\begin{proof}\textquotedblleft{}$\Rightarrow$\textquotedblright{} Fix $T\subseteq X$ absolutely convex and select $N$ as in (B2). For this $N$ select $M$ as in (B1). Fix $n$ and put $T_n:=\Bigcap{k\in\mathbb{N}}(T+{\textstyle\frac{1}{k}}B_{N,n})$. Since $T$ and $B_{N,n}$ are absolutely convex the same is true for $T_n$. Clearly $T_n\subseteq X_N$.  Since $B_{N,n}\subseteq B_{N,n+1}$ we get $T_n\subseteq T_{n+1}$. Accordingly, the set $T_0:=(\Bigcup{n\in\mathbb{N}}T_n)\cap{}X_M$ is an absolutely convex subset of $X_M$.
\smallskip
\\We claim that $T_0$ absorbs $B_{M,m}$ for each $m$. In order to see this, fix $m$ and select $n$ as in (B1). Applying (B2) w.r.t.~the latter $n$ we obtain $S>0$ such that $B_{N,n}\cap X\subseteq ST$. For an arbitrary $k$ we get
$\textstyle B_{N,n}\cap X+\frac{1}{k}B_{N,n}\subseteq ST+\frac{1}{k}B_{N,n}=ST+\frac{S}{Sk}B_{N,n}=S\big(T+\frac{1}{Sk}B_{N,n}\big)$. This yields
$\Bigcap{k\in\mathbb{N}}\big(B_{N,n}\cap X+{\textstyle\frac{1}{k}}B_{N,n}\big)\subseteq S\Bigcap{k\in\mathbb{N}}\big(T+{\textstyle\frac{1}{Sk}}B_{N,n}\big)\subseteq ST_n$. By (B1)  $B_{M,m}\subseteq ST_n$. Therefore $B_{M,m}=B_{M,m}\cap X_M\subseteq ST_n\cap X_M\subseteq S(T_n\cap X_M)\subseteq ST_0$,  and the claim is established.
\smallskip
\\Since $X_M$ is bornological as it is an inductive limit of normed spaces and the sets $B_{M,m}$ form a fundamental system of bounded sets for $X_M$, we conclude that $T_0$ is a 0-neighborhood in $X_M$, hence $T_0\cap X$ is a 0-neighborhood in $X$. To prove that $T$ is a 0-neighborhood, it is enough to show  $T_0\cap X\subseteq 2T$. Let $t\in T_0\cap X$ be given. Then there exists $n$ such that $t\in T_n\cap{}X$. For this $n$ we apply (B2) to get $S>0$ with $B_{N,n}\cap X\subseteq ST$. For $k>S$, $t \in T+\frac{1}{k}B_{N,n}$, i.e.~$t=t_k+\frac{1}{k}b_k$, with $t\in X$, $t_k\in T\subseteq X$ and $b_k\in B_{N,n}$. Thus, $b_k=k(t-t_k) \in X\cap B_{N,n}\subseteq ST$. Therefore $\frac{1}{k}b_k\in T$. Finally we have $x=t_k+{\textstyle\frac{1}{k}}b_k\in T+T\subseteq 2T.$
\medskip
\\\textquotedblleft{}$\Leftarrow$\textquotedblright{} Let $T$ be a 0-neighborhood in $X$. By definition there exist $N$  and a 0-neighbor\-hood $V$ in $X_N$ such that that $V\cap X\subseteq T$. Let $n$ be arbitrary. Since $B_{N,n}$ is bounded in $X_N$, there exists $S>0$ such that $B_{N,n}\subseteq SV$, thus $B_{N,n}\cap X\subseteq ST$.
\end{proof}

Our main result is a direct consequence of Lemma \ref{LEM-1} and the definition of bornological locally convex spaces.

\begin{thm}\label{THM} Let $\mathcal{X}=(X_N,\rho_M^N)$ be a projective spectrum of regular inductive limits of normed spaces $X_N=\ind{n}X_{N,n}$ with inclusions as linking maps and projective limit $X$ satisfying (B1). The space  $X$ is bornological if and only if condition (B2) holds for each absolutely convex and bornivorous set $T\subseteq X$.
\end{thm}

The definition of condition (B1) and the proof of Lemma \ref{LEM-1} were inspired by results of Vogt \cite{VogtLectures, Vogt1989}, see Wengenroth \cite[3.3.4]{Wengenroth}, on the connection of the vanishing of $\Projeins$ for a projective spectrum of LB-spaces and the ultrabornologicity of the corresponding limit. In view of Theorem \ref{THM} and the next proposition, (B1) is in some sense a \textquotedblleft{}weak variant\textquotedblright{} of  condition $\Projeins=0$.

\begin{prop}\label{REM-1} Let $\mathcal{X}$ be a projective spectrum of regular inductive limits of normed spaces. If all $X_{N,n}$ are Banach spaces and $\Projeins \mathcal{X}=0$ holds, then (B1) is satisfied.
\end{prop}
\begin{proof}
We may assume w.l.o.g.~that $(B_{N,n})_{n\in\mathbb{N}}$ is a fundamental system of Banach discs in each of the LB-spaces $X_N$. In the proof of \cite[Theorem 3.3.4]{Wengenroth} it is shown that $\Projeins\mathcal{X}=0$ implies
$$
\forall\:N\;\exists\:M\;\forall\:D\in\mathcal{BD}(X_M)\;\exists\:A\in\mathcal{BD}(X_N)\:\colon D\subseteq \overline{A\cap X}^{(X_N)_A},
$$
where $\mathcal{BD}(X_N)$ is the system of all Banach discs in $X_N$ and $(X_N)_A$ is the Banach space associated to the Banach disc $A$. Now we may replace the Banach disc $A$ by $B_{M,m}$ for some $m$, resp.~$D$ by $B_{N,n}$ for some $n$ and thus the above condition yields
$$
\forall\:N\;\exists\:M\;\forall\:m\;\exists\:n\:\colon B_{M,m}\subseteq \overline{B_{N,n}\cap X}^{X_{N,n}}.
$$
Now (B1) follows, since $\overline{B_{N,n}\cap X}^{X_{N,n}}\subseteq \Bigcap{k\in\mathbb{N}}\big(B_{N,n}\cap X+{\textstyle\frac{1}{k}}B_{N,n}\big)$. \end{proof}

It is well-known that there is a connection between the vanishing of $\Projeins$ on a projective spectrum $\mathcal{X}$ of locally convex spaces and reducedness-properties of the spectrum: If $\mathcal{X}$ is \textit{reduced in the classical sense} (see e.g.~Floret, Wloka \cite[p.~143]{FloretWloka}), i.e.~if the limit space $X$ is dense in each step, then $\mathcal{X}$ is \textit{strongly reduced in the sense of Wengenroth} \cite[Definition 3.3.5]{Wengenroth}, that is for each $N$ there exists $M$ such that $X_M\subseteq\overline{X}^{X_N}$ holds. On the other hand, $\mathcal{X}$ being strongly reduced implies that $\mathcal{X}$ is \textit{reduced in the sense of Wengenroth} \cite[Definition 3.2.17]{Wengenroth}, i.e.~for each $N$ there exists $M$ such that for each $K$ the inclusion $X_M\subseteq \overline{X_K}^{X_N}$ is valid. The latter notion coincides with the one used by Braun, Vogt \cite[Definition 4]{BraunVogt1997}.
\smallskip
\\Wengenroth \cite[remarks previous to Proposition 3.3.8]{Wengenroth} mentioned that for a spectrum $\mathcal{X}$ of LB-spaces $\Projeins\mathcal{X}=0$ implies that $\mathcal{X}$ is strongly reduced. As the next remark shows, for a projective spectrum of inductive limits of normed spaces condition (B1) implies the same property.

\begin{prop}\label{REM-2}  Let $\mathcal{X}=(X_N,\rho_M^N)$ be a projective spectrum of regular inductive limits of normed spaces $X_N=\ind{n}X_{N,n}$ with inclusions as linking maps and projective limit $X$. If  $\mathcal{X}$ satisfies (B1), then $\mathcal{X}$ is strongly reduced, that is for each $N$ there exists $M$ such that $X_M\subseteq\overline{X}^{X_N}$ holds.
\end{prop}
\begin{proof}
For given $N$ we choose $M$ as in (B1) and consider $x\in X_M$. Then there are $m$ and $\rho>0$ with $\rho x\in B_{M,m}$. For this $m$ we apply (B1) to obtain $n$ with $B_{M,m}\subseteq \bar{B_{N,n}\cap X}^{X_{N,n}}$, hence $\rho x\in\bar{B_{N,n}\cap X}^{X_{N,n}}$. Thus there exists $(x_j)_{j\in\mathbb{N}}\subseteq B_{N,n}\cap X$ with $x_j\rightarrow \rho x$ for $j\rightarrow\infty$ w.r.t.~$\|\cdot\|_{N,n}$, thus w.r.t.~the inductive topology of $X_N$. Therefore $ x\in \bar{X}^{X_N}$.
\end{proof}

Roughly speaking the Propositions \ref{REM-1} and \ref{REM-2} mean that condition (B1) is placed \textquotedblleft{}somewhere in between\textquotedblright{} the vanishing of $\Projeins$ and strong reducedness of the spectrum $\mathcal{X}$. In order to be more precise we introduce the following variant of (B1). We say that a spectrum $\mathcal{X}$ satisfies \textit{condition $\text{(}\bar{\text{B1}}\text{)}$} if
$$
\forall\:N\;\exists\:M\;\forall\:m\;\exists\:n\;\forall\:\epsilon>0\;\exists\:B\subseteq X\text{ bounded}\colon B_{M,m}\subseteq B+\epsilon B_{N,n}
$$
holds.
\smallskip
\\Condition $\text{(}\bar{\text{B1}}\text{)}$ is related to the following two conditions of Braun, Vogt \cite[Definition 4]{BraunVogt1997}. We say that $\mathcal{X}$ satisfies $\text{(P}_{\!2}\text{)}$ if
$$
\forall\: N\;\exists\:M,\,n\;\forall\:K,\,m'\;\exists\:k,\,S>0\colon B_{M,m'}\subseteq S(B_{N,n}+B_{K,k}).
$$
We say that $\mathcal{X}$ satisfies $\text{(}\overline{\text{P}_{\!2}}\text{)}$ if
$$
\forall\:N\;\exists\;M',\,n\;\forall\:K,\,m,\,\epsilon>0\:\exists\:k',\,S'>0\colon B_{M',m}\subseteq \epsilon B_{N,n}+S'B_{K,k'}.
$$
Braun, Vogt \cite{BraunVogt1997} proved that for an arbitrary projective spectrum of LB-spaces $\mathcal{X}$, $\Projeins\mathcal{X}=0$ holds if $\mathcal{X}$ satisfies $\text{(}\overline{\text{P}_{\!2}}\text{)}$. Moreover they showed that in the case of a DFS-spectrum $\mathcal{X}$ is reduced and satisfies $\text{(P}_{\!2}\text{)}$ if and only if $\Projeins\mathcal{X}=0$.

\begin{prop}\label{PROP-1} Let $\mathcal{X}=(X_N)_{N\in\mathbb{N}}$ be a projective spectrum of regular LB-spaces with inclusions as linking maps. If $\mathcal{X}$ satisfies $\text{(P}_{\!2}\text{)}$ and $\text{(}\bar{\text{B1}}\text{)}$ then $\mathcal{X}$ satisfies $\text{(}\overline{\text{P}_{\!2}}\text{)}$.
\end{prop}
\begin{proof} $\text{(}\bar{\text{B1}}\text{)}$ can be written as follows
$$
\forall\:M\;\exists\:M'\:\forall\:m\;\exists\:m'\;\forall\:\epsilon>0\;\exists\:B\subseteq X\text{ bounded}\colon B_{M',m}\subseteq B+\epsilon B_{M,m'}.
$$
We show $\text{(}\overline{\text{P}_{\!2}}\text{)}$ in the way it is stated above. Let $N$ be given. We choose $M$ and $n$ as in $\text{(P}_{\!2}\text{)}$ and put $M$ into $\text{(}\bar{\text{B1}}\text{)}$ to obtain $M'$. Let $K$, $m$ and $\epsilon>0$ be given.  We put $m$ into $\text{(}\bar{\text{B1}}\text{)}$ and obtain $m'$.  We put $m'$, $K$ and $\epsilon>0$ into $\text{(P}_{\!2}\text{)}$ and obtain $k$ and $S>0$.  Finally, we put $\frac{\epsilon}{S}$ into $\text{(}\bar{\text{B1}}\text{)}$ and get a bounded set $B\subseteq X$. Now we have by $\text{(}\bar{\text{B1}}\text{)}$ and $\text{(P}_{\!2}\text{)}$ the two inclusions $B_{M',m}\subseteq B+{\textstyle\frac{\epsilon}{S}}B_{M,m'}$ and $B_{M,m'}\subseteq SB_{N,n}+SB_{K,k}$. Since $B$ is bounded in $X$, it is also bounded in the LB-space $X_K$ and this space is regular, i.e.~there exists $k'$ and $\lambda>0$ such that $B\subseteq \lambda B_{K,k'}$ and we clearly may choose $k'\geqslant k$. From the three inclusions we just mentioned we get $B_{M',m}\subseteq (\lambda +\epsilon)B_{K,k'}+\epsilon B_{N,n}$ and thus it is enough to select $S':=\lambda+\epsilon$ to finish the proof.
\end{proof}

For the rest of this section we treat the following special case. We assume $X_{N,n}=X_{N,n+1}=:X_N$ for all $n$ and w.l.o.g.~$B_{N+1}\subseteq B_{N}$, $X=\proj{n}X_N$. We further assume that $X_N$ is a Banach space, thus $X$ is a Fr\'{e}chet space. In this case condition (B1) reduces to
$$
\forall\:N\;\exists\:M\colon B_M\subseteq \Bigcap{k\in\mathbb{N}}B_N\cap X+{\textstyle\frac{1}{k}}B_N
$$
and $\text{(}\bar{\text{B1}}\text{)}$ reduces to
$$
\forall\:N\;\exists\:M\;\forall\:\epsilon>0\;\exists\:B\subseteq X\text{ bounded}\colon B_M\subseteq B+\epsilon B_N.
$$
The latter condition implies
$$
\forall\:N\;\exists\:M\;\forall\:\epsilon>0\;\exists\:B\subseteq X\text{ bounded}\colon B_M \cap X\subseteq B+\epsilon (B_N\cap X),
$$
that is exactly the definition of quasinormability, which was introduced by Gro\-then\-dieck \cite[Definition 4, p.~106 and Lemma 6, p.~107]{Grothendieck1954} (cf.~\cite[Definition after Proposition 26.12]{MeiseVogtEnglisch}) as an extension of  Schwartz spaces and Banach spaces. In fact, a Fr\'{e}chet space is Schwartz if and only if the above condition holds with a finite set $B$, cf.~\cite[Remark previous to 26.13]{MeiseVogtEnglisch}.

\begin{prop}\label{PROP-2} If $\mathcal{X}=(X_N)_{N\in\mathbb{N}}$ is a projective spectrum of Banach spaces with inclusions as linking maps and $X=\proj{N}X_N$ is the corresponding Fr\'{e}chet space, we have (i)$\Rightarrow$(ii)$\Leftrightarrow$(iii) where:\vspace{2pt}
\begin{compactitem}
\item[(i)] Condition $\text{(}\bar{\text{B1}}\text{)}$ holds.\vspace{0pt}
\item[(ii)] $X$ is reduced in the sense $\forall\:N\;\exists\:M\colon X_M\subseteq \bar{X}^{X_N}$.\vspace{2pt}
\item[(iii)] Condition (B1) holds.\vspace{3pt}
\end{compactitem}
In particular \textquotedblleft{}$\text{(}\bar{\text{B1}}\text{)}\Rightarrow\text{(B1)}$\textquotedblright{} holds in general for projective spectra of Banach spaces with inclusions as linking maps.
\end{prop}
\begin{proof}
\textquotedblleft{}(i)$\Rightarrow$(ii)\textquotedblright{} By assumption for each $N$ there is $M$ such that for each $\epsilon>0$ there is a bounded subset $B$ of $X$ with $B_M\subseteq B+\epsilon B_N$. In order to show that $X$ is reduced, we fix $N$ and choose $M$ as in the condition above. Then $B_M\subseteq X+\epsilon B_N$ holds for each $\epsilon>0$ that is $B_M\subseteq \bar{X}^{X_N}$ and thus $X_M\subseteq \bar{X}^{X_N}$.
\smallskip
\\\textquotedblleft{}(ii)$\Rightarrow$(iii)\textquotedblright{} For given $N$ we choose $M>N$ such that $X_M\subseteq \bar{X}^{X_N}$. Let $x\in B_M$. We have $x\in \bar{X}^{X_N}$. Since $B_M\subseteq B_N$ we also have $x\in B_N$. Hence $x\in B_N\cap \bar{X}^{X_N}$. We claim $x\in \bar{B_N\cap X}^{X_N}$. If $x$ is in the interior of $B_{N}$ in $X_N$, we choose a sequence $(x_j)_{j\in\mathbb{N}}\subseteq X$ with $x_j\rightarrow x$ in $X_N$. There exists $J$ such that $x_j\in B_N$ for all $j\geqslant J$. Hence $(x_j)_{j\geqslant J}\subseteq B_N\cap X$ with $x_j\rightarrow x$ in $X_N$ and $x\in\bar{B_N\cap X}^{X_N}$. If otherwise $\|x\|_N=1$, take $(x_j)_{j\in\mathbb{N}}\subseteq X$ with $x_j\rightarrow x$ in $X_N$. We put $y_j:=\frac{x_j}{\|x_j\|_N}$. Then $(y_j)_{j\in\mathbb{N}}\subseteq B_N\cap X$, $y_j\rightarrow \frac{x}{\|x\|_N}=\frac{x}{1}=x$, hence $x\in \bar{B_N\cap X}^{X_N}$.
\smallskip
\\\textquotedblleft{}(iii)$\Rightarrow$(ii)\textquotedblright{} This follows from Proposition \ref{REM-2}.
\smallskip
\\The last statement is now clear.
\end{proof}

\section{Weighted spaces of continuous functions}\label{Applications}

In this section we apply the  criterion in Theorem \ref{THM} to weighted PLB-spaces of continuous functions. The main reference for this section is the article \cite{ABB2009} of Agethen, Bierstedt, Bonet which is an extended and reorganized version of part of the thesis of Agethen \cite{Agethen2004}. In order to present the applications and examples we introduce some notation.
\smallskip
\\Let $X$ denote a locally compact and $\sigma$-compact topological space. By $C(X)$ we denote the space of all continuous functions on $X$ and by $C_c(X)$ the space of all continuous functions on $X$ with compact support. A \textit{weight} is a strictly positive and continuous function on $X$. For a weight $a$ we define the \textit{weighted Banach spaces of continuous functions}
\begin{align*}
Ca(X)&:=\big\{f\in C(X)\:;\:\|f\|_a:=\sup_{x\in X}a(x)|f(x)|<\infty\big\},\\
Ca_0(X)&:=\big\{f\in C(X)\:;\:a|f|\text{ vanishes at } \infty \text{ on } X\big\}.
\end{align*}
Recall that a function $g\colon X\rightarrow \mathbb{R}$ is said to vanish at $\infty$ on $X$ if for each $\epsilon>0$ there is a compact set $K$ in $X$ such that $|g(x)|<\epsilon$ for all $x\in X\backslash{}K$. The space $Ca(X)$ is a Banach space for the norm $\|\cdot\|_a$ and $Ca_0(X)$ is a closed subspace of $Ca(X)$. In the first case we speak of \textit{O-growth conditions} and in the second of \textit{o-growth conditions}.
\smallskip
\\Let now $\mathcal{A}=((a_{N,n})_{N\in\mathbb{N}})_{n\in\mathbb{N}}$ be a double sequence of weights on $X$ which is decreasing in $n$ and increasing in $N$, i.e.~$a_{N,n+1}\leqslant{}a_{N,n}\leqslant{}a_{N+1,n}$ holds for all $N$ and $n$. We define the norms $\|\cdot\|_{N,n}:=\|\cdot\|_{a_{N,n}}$ and get $Ca_{N,n}(X)\subseteq{}Ca_{N,n+1}(X)$ and $C(a_{N,n})_0(X)\subseteq{}C(a_{N,n+1})_0(X)$ with continuous inclusion for each $N$ and $n$. Therefore we can define for each $N$ the \textit{weighted LB-spaces of continuous functions}
$$
\mathcal{A}_NC(X):=\ind{n}Ca_{N,n}(X) \;\; \text{ and } \;\; \mathcal({A}_N)_0C(X):=\ind{n}C(a_{N,n})_0(X).
$$
Since Bierstedt, Bonet \cite[Section 1]{BB1991} implies that the spaces $\mathcal{A}_NC(X)$ are always complete we may assume that every bounded set in $\mathcal{A}_NC(X)$ is contained in $B_{N,n}$ for some $n$ where $B_{N,n}$ denotes the unit ball of $Ca_{N,n}(X)$. The space $(\mathcal{A}_N)_0C(X)$ needs not to be regular. By \cite[Theorem 2.6]{BMS1982} it is regular if and only if it is complete and this is equivalent to the fact that the sequence $\mathcal{A}_N:=(a_{N,n})_{n\in\mathbb{N}}$ is regularly decreasing (see \cite[Definition 2.1 and Theorem 2.6]{BMS1982}). We set $B_{N,n}^{\circ}$ for the unit ball of $C(a_{N,n})_0(X)$. Let us denote by $\mathcal{A}C=(\mathcal{A}_NC(X))_{N}$ and $\mathcal{A}_0C=((\mathcal{A}_N)_0C(X))_{N}$ the projective spectra of LB-spaces where the linking maps are just the inclusions. To complete our definition, we define the \textit{weighted PLB-spaces of continuous functions} by taking projective limits, i.e.~we put
$$
AC(X):=\proj{N}\mathcal{A}_NC(X) \;\; \text{ and } \;\; (AC)_0(X):=\proj{N}\mathcal({A}_N)_0C(X).
$$
By Bierstedt, Meise, Summers \cite[Corollary 1.4.(a)]{BMS1982} $(\mathcal{A}_N)_0C(X)\subseteq\mathcal{A}_NC(X)$ is a topological subspace for each $N$ and hence $(AC)_0(X)$ is a topological subspace of $AC(X)$. Moreover, $\mathcal{A}_0C$ is reduced in the sense that $(AC)_0(X)$ is dense in every step (cf.~\cite[Section 2]{ABB2009}).
\smallskip
\\In \cite{Vogt1992} Vogt introduced the conditions (Q) and (wQ). In the case of weighted PLB-spaces one can reformulate these conditions in terms of the weights as follows. We say that the sequence $\mathcal{A}$ satisfies condition (Q) if
$$
\textstyle\forall \: N \; \exists \: M,\, n \; \forall \: K,\, m,\, \epsilon > 0 \; \exists \: k, \, S>0 : \frac{1}{a_{M,m}} \leqslant \max\big(\frac{\epsilon}{a_{N,n}},\frac{S}{a_{K,k}}\big),
$$
we say that it satisfies (wQ) if
$$
\textstyle\forall \: N \; \exists \: M,\, n \; \forall \: K,\, m \; \exists \: k, \, S>0 : \frac{1}{a_{M,m}} \leqslant \max\big(\frac{S}{a_{N,n}},\frac{S}{a_{K,k}}\big).
$$
It is clear that condition (Q) implies condition (wQ). Bierstedt, Bonet gave in \cite{BB1994}  examples of sequences  satisfying (wQ) but not (Q).
\smallskip
\\One of the main tasks in \cite{ABB2009} was the investigation of locally convex properties of the spaces $AC(X)$ and $(AC)_0(X)$. For this purpose Agethen, Bierstedt, Bonet used the above weight conditions in order to characterize the vanishing of the functor $\Projeins$ on the spectra $\mathcal{A}C$ and $\mathcal{A}_0C$. We state their results.

\begin{thmA}\label{THM-A} (\cite[Theorem 3.7]{ABB2009}) The following conditions are equivalent.\vspace{3pt}
\\\begin{tabular}{rlrl}\vspace{2pt}
(i)  &$\Projeins \mathcal{A}_0C=0$.          & (iii) &$(AC)_0(X)$ is barrelled.               \\
(ii) &$(AC)_0(X)$ is ultrabornological.      & (iv)  &$\mathcal{A}$ satisfies condition (wQ). \\
\end{tabular}
\end{thmA}

\begin{thmB}\label{THM-B} (\cite[Theorems 3.5 and 3.8]{ABB2009}) Consider the following conditions:  \vspace{3pt}
\\\begin{tabular}{rlrl}\vspace{2pt}
(i)  &$\mathcal{A}$ satisfies condition (Q),\hspace{19.8pt} & (iv)  &$AC(X)$ is barrelled,                   \\\vspace{2pt}
(ii) &$\Projeins \mathcal{A}C=0$,                           & (v)   &$\mathcal{A}$ satisfies condition (wQ). \\
(iii)&$AC(X)$ is ultrabornological,                         &                                                \\
\end{tabular}
\smallskip
\\Then (i)$\Leftrightarrow$(ii)$\Rightarrow$(iii)$\Rightarrow$(iv)$\Rightarrow$(v).
\end{thmB}

\subsection{A non-homological proof for the barrelledness of $\text{(AC)}_{\text{0}}\text{(X)}$}

We give an alternative proof of the implication \textquotedblleft{}(iv)$\Rightarrow$(iii)\textquotedblright{} in Theorem A by replacing the machinery of $\Projeins$, which was used in the original proof of Agethen, Bierstedt, Bonet, by a method based on the criterion in Theorem \ref{THM}.

For a given double sequence $\mathcal{A}$ we consider the normed spaces $C(a_{N,n})_c(X):=(C_c(X),\|\cdot\|_{N,n})$ and $(AC)_c(X):=\proj{N}\ind{n}C(a_{N,n})_c(X)$.
We denote by $C_{N,n}$ the closed unit ball of $C(a_{N,n})_c(X)$. Since $C_{N,n} = B_{N,n} \cap C_c(X)$, it follows that $\ind{n}C(a_{N,n})_c(X)$ is a regular inductive limit of normed spaces.
For the proof of Proposition \ref{PROP-3} we need the following technical lemma.

\begin{lem}\label{LEM-3} Let $X=\proj{N}\ind{n}X_{N,n}$ with normed spaces $X_{N,n}$ and let $B_{N,n}$ denote the unit ball of $X_{N,n}$. Let $T\subseteq X$ be absolutely convex and bornivorous and $(n(N))_{N\in\mathbb{N}}\subseteq\mathbb{N}$ be arbitrary. Then there exists $N'\in\mathbb{N}$ such that $\Bigcap{N=1}^{N'}B_{N,n(N)}$ is absorbed by $T$.
\end{lem}
\begin{proof} Assume that the conclusion does not hold. For each $N'$  there is  $x_{N'}\in \Bigcap{N=1}^{N'}B_{N,n(N)}\backslash N'T$. We put $B:=\{x_{N'}\:;\:N'\in\mathbb{N}\}$ and claim that $B$ is bounded in $X$. In order to show this, we fix $L$ and write $B=\{x_{N'}\:;\: 1\leqslant N'\leqslant L\}\cup \{x_{N'}\:;\: N'\geqslant L\}$. To show that $B\subseteq X$ is bounded it is enough to show the latter for $B':=\{x_{N'}\:;\: N'\geqslant L\}$. We claim that $B'\subseteq X_L$ and that $B'$ is bounded there. By definition each $x_{N'}\in B'$ lies in $\Bigcap{N=1}^{N'}B_{N,n(N)}$ and for $L\leqslant N'$ we have $\Bigcap{N=1}^{N'}B_{N,n(N)}\subseteq B_{L,n(L)}$
and the latter set is bounded in $X_L$. Hence the same holds for $B'$ and we have established the claim.  By our assumptions, $T$ is bornivorous. Hence there exists $\lambda>0$ such that $B\subseteq \lambda T$. For $N'\geqslant \lambda$, we get $x_{N'}\not\in N'T\supseteq \lambda T$, a contradiction.
\end{proof}

\begin{prop}\label{PROP-3} The following conditions are equivalent:
\begin{compactitem}\vspace{2pt}
\item[(i)] $\mathcal{A}$ satisfies condition (wQ).\vspace{3pt}
\item[(ii)] $(AC)_c(X)$ is bornological.\vspace{3pt}
\item[(iii)] $(AC)_0(X)$ is barrelled.\vspace{3pt}
\end{compactitem}
\end{prop}
\begin{proof}
\textquotedblleft{}(i)$\Rightarrow$(ii)\textquotedblright{}
By Bierstedt, Bonet \cite{BB1994}, condition (wQ) implies condition $\text{(wQ)}^{\star}$ that is
$$
\exists\: (n(\sigma))_{\sigma\in\mathbb{N}}\subseteq \mathbb{N} \text{ increasing } \forall\:N\;\exists\:M\;\forall\:K,\,m\:\exists\:S>0,\,k\colon
$$
$$
{\textstyle\frac{1}{a_{M,m}}}\leqslant S\max\big({\textstyle\frac{1}{a_{K,k}}},\min_{\scriptscriptstyle\sigma=1,\dots,N}{\textstyle\frac{1}{a_{\sigma,n(\sigma)}}}\big).
$$
Observe that condition (B1) trivially holds for the natural spectrum corresponding to $(AC)_c(X)$. To see that $(AC)_c(X)$ is bornological, we apply Theorem \ref{THM}. It is then enough to show that condition (B2) is satisfied. To see this, fix an absolutely convex and bornivorous set $T\!\subseteq\!(AC)_c(X)$. Since $(AC)_c(X)=C(a_{N,n})_c(X)$ holds algebraically for all $N$, $n$ we may consider $T$ as a subset of the latter space and claim that there exists $N$ such that for each $n$ the ball $C_{N,n}$ is absorbed by $T$. We proceed by contradiction and hence assume that for each $M$ there exists $m(M)$ such that $C_{M,m(M)}$ is not absorbed by $T$. By Lemma \ref{LEM-3}, there exists $N$ such that $\Bigcap{\sigma=1}^{N}C_{\sigma,m(\sigma)}$ is absorbed by $T$. For the sequence $(n(\sigma))_{\sigma\in\mathbb{N}}$ and this $N$ we choose $M$ as in $\text{(wQ)}^{\star}$. Now we put $m=m(M)$ into $\text{(wQ)}^{\star}$. Then for each $K$ there exist $S_K>0$ and $k(K)$ such that $\frac{1}{a_{M,m(M)}}\leqslant S_K\max(\frac{1}{a_{K,k(K)}},\min_{\sigma=1,\dots,N}\frac{1}{a_{\sigma,n(\sigma)}})$ holds. Defining $S_K':=\max_{\mu=1,\dots,K}S_{\mu}$, the latter yields
$$
{\textstyle\frac{1}{a_{M,m(M)}}}\leqslant S_K'\max({\textstyle\min_{\scriptscriptstyle\mu=1,\dots,K}{\textstyle\frac{1}{a_{\mu,k(\mu)}}}},{\textstyle\min_{\scriptscriptstyle\sigma=1,\dots,N}}{\textstyle\frac{1}{a_{\sigma,n(\sigma)}}});
$$
for details we refer to \cite{DISS}. Now an application of the decomposition lemma \cite[Lemma 3.1]{ABB2009} to the above estimate provides that for each $K$ there exists $\tau_K>0$ such that the inclusion $C_{M,m(M)}\subseteq \tau_K[\Bigcap{\sigma=1}^N C_{\sigma,n(\sigma)}+\Bigcap{\mu=1}^KC_{\mu,k(\mu)}]$ is valid. Again we refer to \cite{DISS} for more details. Applying Lemma \ref{LEM-3} a second time, we get $K'$ such that $\Bigcap{\mu=1}^{K'}C_{\mu,k(\mu)}$ is absorbed by $T$. But now in the inclusion $C_{M,m(M)}\subseteq \tau_{K'}[\Bigcap{\sigma=1}^N C_{\sigma,n(\sigma)}+\Bigcap{\mu=1}^{K'}C_{\mu,k(\mu)}]$ the set on the left hand side is not 
absorbed by $T$ unlike the set on the right hand side, a contradiction.
\smallskip
\\\textquotedblleft{}(ii)$\Rightarrow$(iii)\textquotedblright{}
First observe that \cite[Lemma 5.1]{BB1988} implies that $C_c(X)\subseteq(AC)_0(X)$ is a topological subspace, which is dense by \cite[Section 2]{ABB2009}. Therefore $(AC)_0(X)$ is quasibarrelled. Since the latter space is reduced by \cite[Section 2]{ABB2009} it follows from Vogt \cite[Lemma 3.1]{Vogt1989} that it is even barrelled.
\smallskip
\\\textquotedblleft{}(iii)$\Rightarrow$(i)\textquotedblright{} This is Theorem A (Theorem 3.7 in \cite{ABB2009}).
\end{proof}

\subsection{Condition (B1) revisited}

\begin{prop}\label{REM-3}\begin{compactitem}
\item[(a)] If $\mathcal{A}C$ satisfies (B1), then $\mathcal{A}$ satisfies (\underline{Q}),  that is for each $N$ there exists $M$ such that for each $m$ there exists $n$ such that for each $K$ and $\epsilon>0$ there exist $k$ and $S>0$ such that $\frac{1}{a_{M,m}}\leqslant\max(\frac{\epsilon}{a_{N,n}},\frac{S}{a_{K,k}})$ holds.\vspace{3pt}
\item[(b)] If $\mathcal{A}$ satisfies (Q) then the spectrum $\mathcal{A}C$ satisfies condition (B1) and $\mathcal{A}$ satisfies (\underline{Q}).
\end{compactitem}
\end{prop}
\begin{proof}(a) We apply (B1) to conclude $B_{M,m}\subseteq \Bigcap{k\in\mathbb{N}}(B_{N,n}\cap AC(X)+{\textstyle\frac{1}{k}}B_{N,n})\subseteq \Bigcap{k\in\mathbb{N}}(AC(X)+{\textstyle\frac{1}{k}}B_{N,n})
=AC(X)+\Bigcap{k\in\mathbb{N}}{\textstyle\frac{1}{k}}B_{N,n}=AC(X)+\Bigcap{\epsilon>0}\epsilon B_{N,n}$. Now we fix $\epsilon>0$. Since $\frac{1}{a_{M,m}}\in B_{M,m}$, $\frac{1}{a_{M,m}}\in AC(X) +\frac{\epsilon}{2} B_{N,n}$. Thus there exist $f$ and $g$ such that $\frac{1}{a_{M,m}}=f+\frac{\epsilon}{2}g$ with $f\in AC(X)$ and $g\in B_{N,n}$. That is, for each $K$ there exists $k$ and $\lambda>0$ with $|f|\leqslant\frac{\lambda}{a_{K,k}}$ and $|g|\leqslant\frac{1}{a_{N,n}}$.  Then $\textstyle\frac{1}{a_{M,m}}=|f+\epsilon g|\leqslant |f|+\frac{\epsilon}{2}|g|\leqslant \frac{\lambda_K}{a_{K,k}}+\frac{\epsilon}{2a_{N,n}}\;\leqslant\max(\frac{2\epsilon}{2a_{N,n}},
\frac{2\lambda}{a_{K,k}})$, which yields condition (\underline{Q}) with $S:=2\lambda$.
\smallskip
\\(b) By Theorem B, (Q) is equivalent to $\Projeins \mathcal{A}C=0$. Thus Proposition \ref{REM-1} yields that (B1) holds. The implication \textquotedblleft{}$\text{(Q)}\Rightarrow\text{(\underline{Q})}$\textquotedblright{} is clear by definition.
\end{proof}

\begin{prop}\label{REM-4} $\mathcal{A}_0C$ satisfies condition (B1) in general, but even condition (wQ) need not  hold.
\end{prop}
\begin{proof} To prove (B1) it is enough to select $M:=N$ and $n:=m$ and show $B_{N,n}^{\circ}\subseteq\bar{B_{N,n}^{\circ}\cap(AC)_0(X)}^{C(a_{N,n})_0(X)}$. Let $f\in B_{N,n}^{\circ}$, that is $a_{N,n}|f|$ vanishes at $\infty$ and $a_{N,n}|f|\leqslant 1$ on $X$. Define $S_{\alpha}\colon AC(X)\rightarrow(AC)_0(X),\;S_{\alpha}(f)(x):=\alpha(x)\cdot f(x)$, put $A:=\{\,\alpha\in C_c(X)\:;\:0\leqslant\alpha\leqslant1\,\}$, define $\alpha\leqslant\beta:\Leftrightarrow \alpha(x)\leqslant\beta(x)$ for each $x\in X$ and consider the net $(S_{\alpha}f)_{\alpha\in A}$. We have $S_{\alpha}f\in C(a_{N,n})_0(X)$. Since $a_{N,n}|S_{\alpha}f|\leqslant a_{N,n}|f|\leqslant 1$, we have $S_{\alpha}f\in B_{N,n}^{\circ}\cap(AC)_0(X)$. It is easy to see that $S_{\alpha}f\rightarrow f$ w.r.t.~$\|\cdot\|_{N,n}$.
\medskip
\\There are examples of sequences $\mathcal{A}$ which do not satisfy (wQ), cf.~\cite[Example 5.12]{DA}.
\end{proof}

The following result can be regarded as a concrete version of Proposition \ref{PROP-1}. For the proof we introduce the following condition which is inspired by work of Bierstedt, Meise, Summers \cite[Proposition 3.2]{BMS1982a}. The sequence $\mathcal{A}$ satisfies \textit{condition $\text{(}\overline{\text{wS}}\text{)}$} if
$$
\textstyle\forall\:M\;\exists\:M'\;\forall\:m\;\exists\:m'\;\forall\:\epsilon>0\;\exists\:\bar{a}\in \bar{A} \colon \frac{1}{a_{M',m}}\leqslant \bar{a}+\frac{\epsilon}{a_{M,m'}},
$$
where $\bar{A}:=\{\,\bar{a}\colon X\rightarrow\; ]0,\infty[\;;\:\bar{a}\in C(X) \text{ and } \forall\:N\;\exists\:n\colon \sup_{x\in X}a_{N,n}(x)\bar{a}(x)<\infty\,\}$.

\begin{prop}\label{OBS-1} The following conditions are equivalent.
\begin{compactitem}
\item[(i)] $\mathcal{A}$ satisfies condition (wQ) and $\mathcal{A}C$ satisfies (B1).\vspace{3pt}
\item[(ii)] $\mathcal{A}$ satisfies condition (Q).
\end{compactitem}
\end{prop}
\begin{proof}\textquotedblleft{}(i)$\Rightarrow$(ii)\textquotedblright{} Condition (B1) implies
$$
\forall\:M\;\exists\:M'\;\forall\:m\;\exists\:m'\;\forall\:\epsilon>0\colon B_{M',m}\subseteq AC(X)+\epsilon B_{M,m'}.
$$
We show that $\mathcal{A}$ satisfies $\text{(}\overline{\text{wS}}\text{)}$. For given $M$ select $M'$ and for given $m$ select $m'$ as in the condition above. Let $\epsilon>0$ be given. To show the estimate in $\text{(}\overline{\text{wS}}\text{)}$, we consider $\frac{1}{a_{M',m}}\in B_{M',m}$. There exist $a'\in AC(X)$ and $f\in B_{M,m'}$ such that $\frac{1}{a_{M',m}}=a'+\epsilon f$, hence $\frac{1}{a_{M',m}}=\big|\frac{1}{a_{M',m}}\big|\leqslant |a'|+\epsilon|f|\leqslant \bar{a}+\frac{\epsilon}{a_{M,m'}}$,  since $f\in B_{M,m'}$ and by selecting $\bar{a}:=|a'|$. We write (wQ) in the following way
$$
\textstyle\forall\:N\;\exists\:M,\,n\;\forall\:K\,\,m'\;\exists\:k,\,S>0\colon \frac{1}{a_{M,m'}}\leqslant S\big(\frac{1}{a_{N,n}}+\frac{1}{a_{K,k}}\big),
$$
and prove (Q) in the notation
$$
\textstyle\forall\:N\;\exists\:M',\,n\;\forall\:K,\,m\,,\epsilon>0\:\exists\:k',\,S'>0\colon \frac{1}{a_{M',m}}\leqslant\frac{\epsilon}{a_{N,n}}+\frac{S'}{a_{K,k'}}.
$$
Let $N$ be given. We choose $M$ and $n$ as in (wQ). We put $M$ into $\text{(}\overline{\text{wS}}\text{)}$ and obtain $M'$. Let $K$, $m$ and $\epsilon>0$ be given.  We put $m$ into $\text{(}\overline{\text{wS}}\text{)}$ and obtain $m'$. We put $m'$, $K$ and $\epsilon>0$ into (wQ) and obtain $k$ and $S>0$.  Finally, we put $\frac{\epsilon}{S}$ into $\text{(}\overline{\text{wS}}\text{)}$ and obtain $\bar{a}$. Now by (wQ) and $\text{(}\overline{\text{wS}}\text{)}$ we have the two estimates $\frac{1}{a_{M',m}}\leqslant \bar{a}+\frac{\epsilon}{S}\frac{1}{a_{M,m'}}$ and $\frac{1}{a_{M,m'}}\leqslant \frac{S}{a_{N,n}}+\frac{S}{a_{K,k}}$. Moreover, $\bar{a}\in AC(X)$ implies $\bar{a}\in \mathcal{A}_KC(X)$ and hence there exists $k'$ and $\lambda>0$ such that $a_{K,k'}\,\bar{a}\leqslant \lambda$ holds, we  may choose $k'\geqslant k$. Now it is enough to select $S':=\lambda+\epsilon$ in order to get the estimate in (Q).
\smallskip
\\\textquotedblleft{}(ii)$\Rightarrow$(i)\textquotedblright{} Clearly (Q) implies (wQ) and by Proposition \ref{REM-3}.(b), (Q) implies also (B1).
\end{proof}

\begin{cor}\label{SCH-1} If the spectrum $\mathcal{A}C$ satisfies (B1), then it also satisfies condition $(\bar{\text{B1}})$.
\end{cor}
\begin{proof} In the proof of Proposition \ref{OBS-1} we showed that (B1) implies $\text{(}\overline{\text{wS}}\text{)}$, which we may write in the following way
$$\textstyle
\forall\:N\;\exists\:M\;\forall\:m\;\exists\:n\;\forall\epsilon>0\;\exists\:\bar{a}\in \bar{A} \colon \frac{1}{a_{M,m}}\leqslant \bar{a}+\frac{\epsilon}{a_{N,n}}.
$$
To show $\text{(}\bar{\text{B1}}\text{)}$, let $N$ be given. We select $M$ as in $\text{(}\overline{\text{wS}}\text{)}$. For given $m$ we select $n$ as in $\text{(}\overline{\text{wS}}\text{)}$. Let $\epsilon>0$ be given.  We put $\frac{\epsilon}{4}$ into $\text{(}\overline{\text{wS}}\text{)}$ and select $\bar{a}$ as in $\text{(}\overline{\text{wS}}\text{)}$.  Set $B:=\{\,f\in AC(X)\:;\: |f|\leqslant4\bar{a}\,\}$. To show the inclusion in $\text{(}\bar{\text{B1}}\text{)}$ we take  $f\in B_{M,m}$, that is $a_{M,m}|f|\leqslant1$. Then $|f|\leqslant\frac{1}{a_{M,m}}\leqslant \bar{a}+\frac{\epsilon}{2a_{N,n}}\leqslant 2\max(\bar{a},\frac{\epsilon}{4a_{N,n}})=\max(2\bar{a},\frac{\epsilon}{2a_{N,n}})$. According to \cite[Lemma 3.5]{ABB2009} there exist $f_1$, $f_2\in C(X)$ with $f=f_1+f_2$ and $|f_1|\leqslant 2\cdot2\bar{a}$, $|f_2|\leqslant 2\cdot\frac{\epsilon}{2a_{N,n}}$. That is $f_1\in B$ and $f_2\in \epsilon B_{N,n}$,  thus $f\in B + \epsilon B_{N,n}$.
\end{proof}

In view of Theorem B, which provides a characterization of $\Projeins\mathcal{A}C=0$ via (Q) but no characterization of (ultra-)bornological spaces  $AC(X)$, it is a natural question if $\mathcal{A}$ satisfying (wQ) is sufficient for $AC(X)$ being (ultra-)bornological or barrelled. Since this cannot be achieved by the use of $\Projeins$-methods one could hope that the bornologicity criterion (which leaded to a non-homological proof for the implication \textquotedblleft{}(wQ) $\Rightarrow(AC)_0(X)$ barrelled\textquotedblleft{}) would yield an improvement of this type.
Unfortunately this is not the case: Theorem \ref{THM} cannot help us to find any sufficient condition for bornological $AC(X)$ spaces which is strictly weaker than (Q). In fact, if $AC(X)$ is bornological or barrelled, then condition (wQ) follows  by Theorem B. On the other hand, if we wanted to apply  Theorem \ref{THM} we would have to assume (B1) and  by Proposition \ref{OBS-1} the sequence $\mathcal{A}$ must satisfy (Q).

\subsection{The case of Fr\'{e}chet spaces}

We study the case that the spaces $AC(X)$ and $(AC)_0(X)$ are Fr\'{e}chet spaces. That is, we put $a_{N,n}=2^na_N$ for some increasing sequence $(a_N)_{N\in\mathbb{N}}$. Alternatively, we may simply define $AC(X)=\proj{N}Ca_N(X)$ and $(AC)_0(X)=\proj{N}C(a_N)_0(X)$.
\smallskip
\\Before we present results on the above spaces for a general locally compact and $\sigma$-compact space $X$ let us study the case $X=\mathbb{N}$. In this situation, the spaces under consideration turn out to be the well-known K\"othe echelon spaces $\lambda^{\infty}(A)$ and $\lambda^0(A)$ where the K\"othe matrix $A$ is given by $A=(a_N)_{N\in\mathbb{N}}$ (in the notation of \cite[Definition 1.2]{BMS1982a}).
\smallskip
\\The following observations are easy; they all refer to the case that the spaces $AC(X)$ and $(AC)_0(X)$ are Fr\'{e}chet spaces and that $X=\mathbb{N}$.
\begin{compactitem}
\item[a.] The system $\bar{A}$ introduced in the proof of Proposition \ref{OBS-1} is just the K\"othe set
$$
\bar{V}=\big\{\bar{a}\colon \mathbb{N}\rightarrow\;\:]0,\infty[\:;\:\forall\:N\colon\sup_{i\in\mathbb{N}}a_N(i)\bar{a}(i)<\infty\big\}
$$
of Bierstedt, Meise, Summers \cite[Definition 1.4]{BMS1982a}.\vspace{3pt}
\item[b.] Condition $(\overline{\text{wS}})$ of the proofs of Proposition \ref{OBS-1} and Corollary \ref{SCH-1} reduces to
$$
\textstyle\forall\:N\;\exists\:M\;\forall\:\epsilon>0\;\exists\:\bar{a}\in \bar{A}\: \colon \frac{1}{a_{M}}\leqslant \bar{a}+\frac{\epsilon}{a_{N}},
$$
which is equivalent to condition
$$
\text{(wS)}\;\;\;\;\;\begin{array}{c}\vspace{3pt}
\forall\:N\;\exists\:M\;\forall\:\epsilon>0\;\exists\:\bar{a}\in\bar{A}\;\:\forall\:i\in\mathbb{N}\colon\\
\textstyle \frac{1}{a_M(i)}\leqslant\frac{\epsilon}{a_N(i)} \: \text{ whenever } \: \bar{a}(i)<\frac{1}{a_M(i)}
\end{array}
$$
of Bierstedt, Meise, Summers \cite[Proposition 3.2]{BMS1982a}.\vspace{3pt}
\item[c.] The conditions (Q) and (\underline{Q}) both are equivalent to
$$
\textstyle\forall\:N\;\exists\:M\;\forall\:K,\,\epsilon>0\;\exists\: S>0\colon \frac{1}{a_M}\leqslant\frac{\epsilon}{a_N}+\frac{S}{a_K}.
$$
\end{compactitem}

They are also equivalent to the regularly decreasing condition of \cite{BMS1982}.

Let us now review some well-known results on the spaces $\lambda^{\infty}(A)$ and $\lambda^0(A)$, which should be compared with Propositions \ref{PROP-5} and \ref{PROP-6} below.

\begin{prop}\label{REM-5}\footnote{see Bierstedt, Meise, Summers \cite[Proposition on p.~48, Proposition 3.2, Corollary 3.5 and Example 3.11]{BMS1982a}, Vogt \cite[last Remark on page 167]{Vogt1987} and Meise, Vogt \cite[27.20]{MeiseVogtEnglisch}.} Let $A$ be a K\"othe matrix and denote by $\mathcal{A}L$ and $\mathcal{A}_0L$ the natural projective spectra corresponding to $\lambda^{\infty}(A)$ and $\lambda^{0}(A)$, respectively. 
\begin{compactitem}
\item[(a)] The following conditions are equivalent.
\begin{compactitem}\vspace{2pt}
\item[(i)] $\mathcal{A}L$ is reduced.\vspace{2pt}
\item[(ii)] $\lambda^{\infty}(A)$ is quasinormable.\vspace{2pt}
\item[(iii)] $A$ satisfies condition (wS).\vspace{2pt}
\item[(iv)] $A$ satisfies condition (Q).\vspace{2pt}
\item[(v)] $A$ satisfies condition (\underline{Q}).
\end{compactitem}\vspace{3pt}
\item[(b)] $\mathcal{A}_0L$ is always reduced. Moreover, the following conditions are equivalent.
\begin{compactitem}\vspace{2pt}
\item[(i)] $\lambda^{0}(A)$ is quasinormable.\vspace{2pt}
\item[(ii)] $A$ satisfies condition (wS).\vspace{2pt}
\item[(iii)] $A$ satisfies condition (Q).\vspace{2pt}
\item[(iv)] $A$ satisfies condition (\underline{Q}).
\end{compactitem}\vspace{3pt}
\item[(c)] There exists a K\"othe matrix $A$ which does not satisfy condition (wS), that is the space $\lambda^0(A)$ is reduced but not quasinormable.
\end{compactitem}
\end{prop}

As a consequence the implication \textquotedblleft{}(ii)$\Rightarrow$(i)\textquotedblright{} in Proposition \ref{PROP-2}  and the implication \textquotedblleft{}(B1)$\Rightarrow(\bar{\text{B1}})$\textquotedblright{} do not hold in general.
\smallskip
\\To conclude, we consider Fr\'{e}chet spaces $AC(X)$ and $(AC)_0(X)$ for an arbitrary locally compact and $\sigma$-compact topological space $X$.

\begin{prop}\label{PROP-5} In the Fr\'{e}chet case, the following conditions are equivalent.
\vspace{3pt}
\\\begin{tabular}{rlrl}
(i)   &$AC(X)$ is quasinormable.      & (v)   &$\mathcal{A}$ satisfies (Q).             \\
(ii)  &$\mathcal{A}C$ is reduced.            & (vi)  &$\mathcal{A}$ satisfies (\underline{Q}).    \\
(iii) &$\mathcal{A}C$ satisfies (B1).         & (vii) &$\mathcal{A}$ satisfies condition $\text{(}\overline{\text{wS}}\text{)}$. \\
(iv)  &$\mathcal{A}C$ satisfies $\text{(}\bar{\text{B1}}\text{)}$.&       & \\
\end{tabular}
\end{prop}
\begin{proof} \textquotedblleft{}(iv)$\Rightarrow$(ii)\textquotedblright{} This is Proposition \ref{PROP-2}.
\smallskip
\\\textquotedblleft{}(ii)$\Rightarrow$(iii)\textquotedblright{} This is Proposition \ref{PROP-2}.
\smallskip
\\\textquotedblleft{}(iii)$\Rightarrow$(iv)\textquotedblright{} This is Corollary \ref{SCH-1}.
\smallskip
\\\textquotedblleft{}(iv)$\Rightarrow$(i)\textquotedblright{} As we noted before Proposition \ref{PROP-2}, for projective spectra of Banach spaces $\text{(}\bar{\text{B1}}\text{)}$ implies the definition of quasinormability.
\smallskip
\\\textquotedblleft{}(i)$\Leftrightarrow$(vii)\textquotedblright{} This follows from Bierstedt, Meise \cite[Proof of Proposition 5.8]{BM1986}.
\smallskip
\\\textquotedblleft{}(vii)$\Leftrightarrow$(v)\textquotedblright{} This is known; see Proposition \ref{REM-5}.
\smallskip
\\\textquotedblleft{}(v)$\Leftrightarrow$(vi)\textquotedblright{} As we noted before Proposition \ref{REM-5}, in the Fr\'{e}chet case (Q) and (\underline{Q}) coincide.
\smallskip
\\\textquotedblleft{}(v)$\Rightarrow$(iii)\textquotedblright{} This is Proposition \ref{REM-3}.(b).
\smallskip
\\\textquotedblleft{}(iii)$\Rightarrow$(v)\textquotedblright{} In the Fr\'{e}chet case condition (wQ) reduces to
$$
\textstyle\forall \:N\;\exists\:M\;\forall\:K\;\exists\:S>0\colon\frac{1}{a_{M}} \leqslant S\max\big(\frac{1}{a_{N}},\frac{1}{a_{K}}\big)
$$
and is always satisfied: Let $N$ be given. We choose $M:=N$. For given $K$ we put $S:=1$. Then the estimate $\frac{1}{a_N}\leqslant\max(\frac{1}{a_{N}},\frac{1}{a_K})$ is trivial. Hence, Proposition \ref{OBS-1} yields the desired implication.

\end{proof}

\begin{prop}\label{PROP-6} In the Fr\'{e}chet case, the following statements hold.
\begin{compactitem}
\item[(i)] $\mathcal{A}_0C$ is always reduced.\vspace{2pt}
\item[(ii)] (wQ) is always satisfied.\vspace{2pt}
\item[(iii)] For $\mathcal{A}_0C$, condition (B1) is always satisfied.\vspace{2pt}
\item[(iv)] $(AC)_0(X)$ fails to be quasinormable in general. Thus conditions (B1) and $\text{(}\bar{\text{B1}}\text{)}$ are not equivalent for $\mathcal{A}_0C$.
\end{compactitem}
\end{prop}
\begin{proof}
(i) This follows from Agethen, Bierstedt, Bonet \cite[Section 2]{ABB2009}.
\smallskip
\\(ii) See the proof of \textquotedblleft{}(iii)$\Rightarrow$(v)\textquotedblleft{} in Proposition \ref{PROP-5}.
\smallskip
\\(iii) By Proposition \ref{PROP-2}, (B1) is equivalent to the reducedness of $(AC)_0(X)$. Hence the assertion follows from statement (i).
\smallskip
\\(iv) This follows from Proposition \ref{REM-5}.(c). Now, it is enough to recall that for projective spectra of Banach spaces $\text{(}\bar{\text{B1}}\text{)}$ implies the definition of quasinormability.
\end{proof}

\noindent
\textbf{Acknowledgement.} The authors are indebted to P.\ Doma\'nski who pointed out an error in the former statement of Proposition \ref{PROP-2}. The research of J.\ Bonet was partially supported by by MEC and FEDER Project MTM 2007-62643; GV Project Prometeo 2008/101 and the net MTM 2007-30904-E (Spain).

\setlength{\parskip}{0cm}

\vspace{10pt}

\begin{minipage}{7.5cm}
J.\ Bonet

Instituto Universitario de Matem\'{a}tica Pura

y Aplicada IUMPA

Universidad Polit\'{e}cnica de Valencia

E-46071 Valencia

SPAIN

e-mail: jbonet@mat.upv.es
\end{minipage}
\hspace*{\fill}
\parbox{7.5cm}{S.-A.\ Wegner

FB C -- Mathematik

Bergische Universit\"at Wuppertal

Gau\ss{}str.~20

D-42119 Wuppertal

GERMANY

e-mail: wegner@uni-wuppertal.de

}


\small


\begin{thebibliography}{10}

\bibitem{Agethen2004}
S.~Agethen, \emph{Spaces of continuous and holomorphic functions with growth
  conditions}, Dissertation Universit\"{a}t Paderborn, 2004.

\bibitem{ABB2009}
S.~Agethen, K.~D. Bierstedt, and J.~Bonet, \emph{Projective limits of weighted
  ({LB})-spaces of continuous functions}, Arch. Math. (Basel) \textbf{92}
  (2009), no.~5, 384--398.

\bibitem{Bierstedt2001}
K.~D. Bierstedt, \emph{A survey of some results and open problems in weighted
  inductive limits and projective description for spaces of holomorphic
  functions}, Bull. Soc. R. Sci. Li\`{e}ge \textbf{70} (2001) no. 4--6, 167--182.

\bibitem{BB1988}
K.~D. Bierstedt and J.~Bonet, \emph{Dual density conditions in (DF)-spaces, II}, Bull. Soc. Roy. Sci. Li\`{e}ge  \textbf{57} (1988), 567--589.


\bibitem{BB1991}
K.~D. Bierstedt and J.~Bonet, \emph{Completeness of the ({LB})-spaces
  {$\mathcal{V}C(X)$}}, Arch. Math. (Basel) \textbf{56} (1991), no.~3,
  281--285.

\bibitem{BB1994}
K.~D. Bierstedt and J.~Bonet, \emph{Weighted ({LF})-spaces of continuous functions}, Math. Nachr.
  \textbf{165} (1994), 25--48.

\bibitem{BM1986}
K.~D. Bierstedt and R.~Meise, \emph{Distinguished echelon spaces and the
  projective description of weighted inductive limits of type
  $\mathcal{V}_d\mathcal{C}({X})$}, Aspects of Mathematics and its
  Applications, North-Holland Math. Library \textbf{34} (1986), 169--2226.

\bibitem{BMS1982a}
K.~D. Bierstedt, R.~Meise, and W.~H. Summers, \emph{K\"othe sets and {K}\"othe
  sequence spaces}, Functional analysis, holomorphy and approximation theory
  ({R}io de {J}aneiro, 1980), North-Holland Math. Stud., vol.~71,
  North-Holland, Amsterdam, 1982, pp.~27--91.

\bibitem{BMS1982}
K.~D. Bierstedt, R.~Meise, and W.~H. Summers, \emph{A projective description of weighted inductive limits}, Trans.
  Amer. Math. Soc. \textbf{272} (1982), no.~1, 107--160.

\bibitem{BraunVogt1997}
R.~Braun and D.~Vogt, \emph{A sufficient condition for \rm
  {P}roj$\,^1\mathcal{X}=0$}, Michigan Math. J. \textbf{44} (1997), no.~1,
  149--156.

\bibitem{Domanski2004}
P.\ Doma{\'n}ski, \emph{Classical {PLS}-spaces: spaces of distributions,
  real analytic functions and their relatives}, Orlicz centenary volume, Banach
  Center Publ., vol.~64, Polish Acad. Sci., Warsaw, 2004, pp.~51--70.

\bibitem{FloretWloka}
K.~Floret and J.~Wloka, \emph{Einf\"uhrung in die {T}heorie der lokalkonvexen
  {R}\"aume}, Lecture Notes in Mathematics, No. 56, Springer-Verlag, Berlin,
  1968.

\bibitem{Grothendieck1954}
A.~Grothendieck, \emph{Sur les espaces ($\mathcal{F}$) et ($\mathcal{DF}$)},
  Summa Brasil. Math. \textbf{3} (1954), 57--123.

\bibitem{Jarchow}
H.~Jarchow, \emph{Locally {C}onvex {S}paces}, B. G. Teubner, Stuttgart, 1981.

\bibitem{KoetheI}
G.~K{\"o}the, \emph{Topological vector spaces. {I}}, Translated from the German
  by D. J. H. Garling. Die Grundlehren der mathematischen Wissenschaften, Band
  159, Springer-Verlag New York Inc., New York, 1969.

\bibitem{KoetheII}
G.~K{\"o}the, \emph{Topological vector spaces. {II}}, Grundlehren der Mathematischen
  Wissenschaften [Fundamental Principles of Mathematical Science], vol. 237,
  Springer-Verlag, New York, 1979.

\bibitem{MeiseVogtEnglisch}
R.~Meise and D.~Vogt, \emph{Introduction to functional analysis}, Oxford
  Graduate Texts in Mathematics, vol.~2, The Clarendon Press Oxford University
  Press, New York, 1997, Translated from the German by M. S. Ramanujan and
  revised by the authors.

\bibitem{Palamodov1971}
V.~P. Palamodov, \emph{Homological methods in the theory of locally convex
  spaces}, Uspekhi Mat. Nauk \textbf{26} (1) (1971) 3--66 (in Russian), English
  transl., Russian Math. Surveys \textbf{26} (1) (1971), 1--64.

\bibitem{Palamodov1968}
V.~P. Palamodov, \emph{The projective limit functor in the category of topological
  linear spaces}, Mat. Sb. \textbf{75} (1968) 567--603 (in Russian), English
  transl., Math. USSR Sbornik \textbf{17} (1972), 189--315.

\bibitem{BPC}
P.~{P\'{e}rez Carreras} and J.~Bonet, \emph{Barrelled {L}ocally {C}onvex
  {S}paces}, North--Holland Mathematics Studies \textbf{113}, 1987.

\bibitem{VogtLectures}
D.~Vogt, \emph{Lectures on projective spectra of ({DF})-spaces}, Seminar
  lectures, {W}uppertal, 1987.

\bibitem{Vogt1987}
D.~Vogt, \emph{On the functors {${\rm Ext}\sp 1(E,F)$} for {F}r\'echet spaces},
  Studia Math. \textbf{85} (1987), no.~2, 163--197.

\bibitem{Vogt1989}
D.~Vogt, \emph{Topics on projective spectra of ({LB})-spaces}, Adv. in the
  {T}heory of {F}r\'{e}chet spaces (Istanbul, 1988), NATO Adv. Sci. Inst. Ser.
  C \textbf{287} (1989), 11--27.

\bibitem{Vogt1992}
D.~Vogt, \emph{Regularity properties of ({LF})-spaces}, Progress in Functional
  Analysis, Proc. Int. Meet. Occas. 60th Birthd. M. Valdivia,
  Pe\~{n}\'{\i}scola/Spain, North-Holland Math. Stud. \textbf{170} (1992), 57--84.

\bibitem{DA}
S.-A. Wegner, \emph{Inductive kernels and projective hulls for weighted
  ({PLB})- and ({LF})-spaces of continuous functions}, Diplomarbeit Universit\"{a}t
  Paderborn, 2007.

\bibitem{DISS}
S.-A. Wegner, \emph{Projective limits of weighted ({LB})-spaces of holomorphic
  functions}, PhD-thesis, Universidad Polit\'{e}cnica de Valencia, in
  preparation, 2010.

\bibitem{Wengenroth}
J.~Wengenroth, \emph{Derived {F}unctors in {F}unctional {A}nalysis}, Lecture
  {N}otes in {M}athematics \textbf{1810}, Springer, Berlin, 2003.

\end{thebibliography}
\end{document}